\documentclass[11pt]{article}

\usepackage[margin=1.2in]{geometry}

\usepackage[all]{xy}
\usepackage{amsmath}
\usepackage{amsthm}
\usepackage{graphicx}
\usepackage{color}
\usepackage{amsfonts}
\usepackage{amssymb}
\usepackage{amscd}
\usepackage{url}

\newcommand{\re}{\mathbb{R}}
\newcommand{\cpx}{\mathbb{C}}

\newcommand{\N}{\mathbb{N}}

\newcommand{\half}{\frac{1}{2}}
\newcommand{\lmd}{\lambda}

\newcommand{\eps}{\epsilon}

\def\af{\alpha}

\def\rank{\mbox{rank}}

\newcommand{\sig}{\sigma}

\newcommand{\reff}[1]{(\ref{#1})}
\newcommand{\pt}{\partial}

\newcommand{\bdes}{\begin{description}}
\newcommand{\edes}{\end{description}}

\newcommand{\bal}{\begin{align}}
\newcommand{\eal}{\end{align}}

\newcommand{\bnum}{\begin{enumerate}}
\newcommand{\enum}{\end{enumerate}}

\newcommand{\bit}{\begin{itemize}}
\newcommand{\eit}{\end{itemize}}

\newcommand{\bea}{\begin{eqnarray}}
\newcommand{\eea}{\end{eqnarray}}
\newcommand{\be}{\begin{equation}}
\newcommand{\ee}{\end{equation}}

\newcommand{\baray}{\begin{array}}
\newcommand{\earay}{\end{array}}

\newcommand{\bsry}{\begin{subarray}}
\newcommand{\esry}{\end{subarray}}

\newcommand{\bca}{\begin{cases}}
\newcommand{\eca}{\end{cases}}

\newcommand{\bcen}{\begin{center}}
\newcommand{\ecen}{\end{center}}

\newcommand{\bbm}{\begin{bmatrix}}
\newcommand{\ebm}{\end{bmatrix}}

\newcommand{\bmx}{\begin{matrix}}
\newcommand{\emx}{\end{matrix}}

\newcommand{\bpm}{\begin{pmatrix}}
\newcommand{\epm}{\end{pmatrix}}

\newcommand{\btab}{\begin{tabular}}
\newcommand{\etab}{\end{tabular}}

\newtheorem{theorem}{Theorem}[section]

\newtheorem{lem}[theorem]{Lemma}

\newtheorem{cor}[theorem]{Corollary}

\newtheorem{ass}[theorem]{Assumption}

\theoremstyle{definition}

\newtheorem{exm}[theorem]{Example}

\begin{document}

\title{ \centerline{An Exact Jacobian SDP Relaxation for Polynomial Optimization}
\author{Jiawang Nie\footnote{Department of Mathematics,
University of California, 9500 Gilman Drive, La Jolla, CA 92093.
Email: njw@math.ucsd.edu. The research was partially supported by NSF grants
DMS-0757212, DMS-0844775 and Hellman Foundation Fellowship.}}
\date{June 10, 2010}
}

\maketitle

\begin{abstract}
Given polynomials $f(x), g_i(x), h_j(x)$,
we study how to minimize $f(x)$ on the set
\[
S = \left\{ x \in \re^n:\,
h_1(x) = \cdots = h_{m_1}(x) = 0,
g_1(x) \geq 0, \ldots, g_{m_2}(x) \geq 0
\right\}.
\]
Let $f_{min}$ be the minimum of $f$ on $S$.
Suppose $S$ is nonsingular and $f_{min}$ is achievable on $S$, 
which is true generically.
The paper proposes a new semidefinite programming (SDP) relaxation for this problem.
First we construct a set of new polynomials $\varphi_1(x), \ldots, \varphi_r(x)$,
by using the Jacobian of $f,h_i,g_j$, such that the above problem
is equivalent to
\[
\baray{rl}
\underset{x\in\re^n}{\min} & f(x) \\
\mbox{s.t.} & h_i(x) = 0, \, \varphi_j(x) = 0,
\, 1\leq i \leq m_1, 1 \leq j \leq r, \\
& g_1(x)^{\nu_1} \cdots g_{m_2}(x)^{\nu_{m_2}} \geq 0, \,
\quad \forall \nu \,\in \{0,1\}^{m_2}.
\earay
\]
Then we prove that for all $N$ big enough,
the standard $N$-th order Lasserre's SDP relaxation is exact
for solving this equivalent problem, that is,
it returns a lower bound that is equal to $f_{min}$.
Some variations and examples are also shown.
\end{abstract}



\section{Introduction}

Consider the optimization problem
\be  \label{pop:gen}
\baray{rl}
\underset{x\in \re^n}{\min} & f(x) \\
\mbox{s.t.} & h_1(x) = \cdots = h_{m_1}(x) = 0 \\
& g_1(x) \geq 0, \ldots, g_{m_2}(x) \geq 0
\earay
\ee
where $f(x), g_i(x), h_j(x)$ are polynomial functions in $x\in \re^n$.
Let $S$ be its feasible set and
$f_{min}$ be its global minimum.
We are interested in finding $f_{min}$.
This problem is NP-hard, even when one of the polynomials is quadratic.

A standard approach for solving \reff{pop:gen} is semidefinite programming (SDP)
relaxations proposed by Lasserre \cite{Las01}.
It is based on a sequence of sum of squares (SOS) type representations
of polynomials that are nonnegative on $S$.
The basic idea is, for a given integer $N>0$ (called relaxation order),
solve the SOS program
\be  \label{sos:Put}
\baray{rl}
\max & \gamma \\
\mbox{\mbox{s.t.}} & f(x) - \gamma =
\overset{m_1}{ \underset{ i=1}{\sum} }  \phi_i(x) h_i(x)
+ \overset{m_2}{ \underset{ j=1}{\sum} }   \sig_j(x) g_j(x), \\
& \deg (\phi_i h_i ), \, \deg( \sig_j g_j) \leq 2N \,\quad
\forall \, i, j, \\
& \sig_1,\ldots, \sig_{m_2}  \mbox{ are SOS.}
\earay
\ee
In the above, $g_0(x) \equiv 1$, the decision variables are the coefficients of
polynomials $\phi_i$ and $\sig_j$.
Here a polynomial is SOS if it is a sum of squares of other polynomials.
The SOS program \reff{sos:Put} is equivalent to an SDP problem (see \cite{Las01}).
We refer to \cite{ParMp,PS03} for more about the connection between SDP and SOS programs.
Let $p_N$ be the optimal value of \reff{sos:Put}.
Clearly, $p_N\leq f_{min}$ for every $N$.
Using Putinar's Positivstellensatz \cite{Put}, Lasserre proved $p_N \to f_{min}$
as $N\to \infty$, under the archimedean condition.
A stronger relaxation than \reff{pop:gen} would be
obtained by using cross products of $g_j$, which is
\be  \label{sos:Smg}
\baray{rl}
\max & \gamma \\
\mbox{\mbox{s.t.}} & f(x) - \gamma =
\underset{ i=1,\ldots,m_1 }{\sum}  \phi_i(x) h_i(x)
+ \underset{ \nu \in \{0,1\}^{m_2} }{\sum}
\sig_\nu(x) \cdot g_\nu(x), \\
& \deg (\phi_i h_i ) \leq 2N, \,
\deg( \sig_\nu g_\nu) \leq 2N  \quad \forall \, i, \nu, \\
&  \sig_\nu  \mbox{ are all SOS}.
\earay
\ee
Here, denote $g_\nu=g_1^{\nu_1}\cdots g_{m_2}^{\nu_{m_2}}$.
Let $q_N$ be the optimal value of \reff{sos:Smg}.
When $S$ is compact,
Lasserre showed $q_N \to f_{min}$ as $N$ goes to infinity,
using Schm\"{u}gen's Positivstellensatz \cite{Smg}.
An analysis for the convergence speed
of $p_N, q_N$ to $f_{min}$ is given in \cite{NS07,Swg04}.
Typically, \reff{sos:Put} and \reff{sos:Smg} are not exact for \reff{pop:gen} with a finite $N$.
Scheiderer \cite{Sch99} proved a very surprising result:
whenever $S$ has dimension three or higher, there always exists $f$
such that $f(x)-f_{min}$ does not have a representation required in \reff{sos:Smg}.
Thus, we usually need solve a big number of SDPs until convergence is met.
This is very inefficient in many applications.
Furthermore, when $S$ is not compact,
typically we do not have the convergence of $p_N, q_N$ to $f_{min}$.
This is another difficulty.
Thus, people are interested in more efficient methods for solving \reff{pop:gen}.

Recently, the author, Demmel and Sturmfels \cite{NDS} proposed a gradient type SOS relaxation.
Consider the case of \reff{pop:gen} without constraints. If the minimum $f_{min}$ is
achieved at a point $u$, then $\nabla f(u)=0$, and the problem is equivalent to
\be \label{po:grad=0}
\underset{x\in \re^n}{\min} \quad f(x)  \quad
s.t. \quad \frac{\pt f}{\pt x_1} = \cdots = \frac{\pt f}{\pt x_n} = 0.
\ee
In \cite{NDS}, Lasserre's relaxation is applied to solve \reff{po:grad=0}.
It was shown in \cite{NDS} that a sequence of lower bounds converging to $f_{min}$
would be obtained, and it has finite convergence
if the gradient ideal, generated by the partial derivatives of $f(x)$, is radical.
More recently, Demmel, the author and Powers \cite{DNP}
generalized the gradient SOS relaxation to solve \reff{pop:gen}
by using the Karush-Kuhn-Tucker (KKT) conditions of \reff{pop:gen}
\[
\nabla f(x) = \overset{m_1}{ \underset{i=1}{\sum} } \lmd_i \nabla h_i(x)
+ \overset{m_2}{ \underset{j=1}{\sum} }   \mu_j \nabla g_j(x), \quad
\mu_j g_j(x) = 0, \, j=1,\ldots,m_2.
\]
If a global minimizer of \reff{pop:gen} is a KKT point,
then  \reff{pop:gen} is equivalent to
\be \label{pop:+kkt}
\baray{rl}
\underset{x,\lmd,\mu}{\min} & f(x) \\
\mbox{s.t.} & h_1(x) = \cdots = h_{m_1}(x) = 0, \\
&\nabla f(x) = \overset{m_1}{ \underset{i=1}{\sum} } \lmd_i \nabla h_i(x)
+ \overset{m_2}{ \underset{j=1}{\sum} }   \mu_j \nabla g_j(x), \\
& \mu_j g_j(x) =  0, \quad g_j(x) \geq 0, \, j = 1, \ldots, m_2.
\earay
\ee
Let $\{v_N\}$ be the sequence of lower bounds for \reff{pop:+kkt} obtained
by applying Lasserre's relaxation of type \reff{sos:Smg}.
It was shown in \cite{DNP} that $v_N\to f_{min}$, no matter $S$ is compact or not.
Furthermore, it holds that $v_N=f_{min}$ for a finite $N$
when the KKT ideal is radical,
but it was unknown in \cite{DNP} whether this property still holds
without the KKT ideal being radical.
A drawback for this approach is that the involved polynomials are in $(x,\lmd,\mu)$.
There are totally $n+m_1+m_2$ variables, which makes
the resulting SDP very difficult to solve in practice.

\bigskip \noindent
{\bf Contributions} \,
This paper proposes a new SDP type relaxation for solving \reff{pop:gen}
using KKT conditions but the involved polynomials are only in $x$.
Suppose $S$ is nonsingular and $f_{min}$ is achievable,
which is true generically.
We construct a set of new polynomials $\varphi_1(x), \ldots, \varphi_r(x)$,
by using the minors of the Jacobian of $f,h_i,g_j$, such that \reff{pop:gen}
is equivalent to 
\[
\baray{rl}
\underset{x\in\re^n}{\min} & f(x) \\
\mbox{s.t.} &  h_i(x) = \varphi_j(x) = 0,
\, 1\leq i \leq m_1, 1 \leq j \leq r, \\
& g_1(x)^{\nu_1} \cdots g_{m_2}(x)^{\nu_{m_2}} \geq 0,
\quad \forall \, \nu \in \{0,1\}^{m_2}.
\earay
\]
Then we prove that for all $N$ big enough,
the standard $N$-th order Lasserre's relaxation for the above returns a lower bound
that is equal to the minimum $f_{min}$.
That is, an exact SDP relaxation for \reff{pop:gen} is obtained
by using the Jacobian.

This paper is organized as follows.
Section~\ref{sec:constr} gives the construction of
this exact SDP relaxation by using Jacobian.
Its exactness and genericity are proved in Section~\ref{sec:proof}.
Some efficient variations are proposed in Section~\ref{sec:var}.
Some examples of how to apply this exact SDP relaxation
are shown in Section~\ref{sec:exmp}.
Some conclusions and discussions are made in Section~\ref{sec:con-dis}.
Finally, we attach an appendix
introducing some basics of algebraic geometry and real algebra
that are used in the paper.

\bigskip \noindent
{\bf Notations} \,
The symbol $\N$ (resp., $\re$, $\cpx$) denotes the set of
nonnegative integers (resp., real numbers, complex numbers).
For any $t\in \re$, $\lceil t\rceil$
denotes the smallest integer not smaller than $t$.
For integer $n>0$, $[n]$ denotes the set $\{1,\ldots,n\}$,
and $[n]_k$ denotes the set of subsets of $[n]$ whose cardinality is $k$.
For a subset $J$ of $[n]$, $|J|$ denotes its cardinality.
For $x \in \re^n$, $x_i$ denotes the $i$-th component of $x$,
that is, $x=(x_1,\ldots,x_n)$.
For $\af \in \N^n$, denote $|\af| = \af_1 + \cdots + \af_n$.
For $x \in \re^n$ and $\af \in \N^n$,
$x^\af$ denotes $x_1^{\af_1}\cdots x_n^{\af_n}$.
The symbol $\re[x] = \re[x_1,\ldots,x_n]$ (resp. $\cpx[x] = \cpx[x_1,\ldots,x_n]$)
denotes the ring of polynomials in $(x_1,\ldots,x_n)$ with real (resp. complex) coefficients.
A polynomial is called a form if it is homogeneous.
The $\re[x]_{\leq d}$ denotes the subspace of polynomials
in $\re[x]$ of degrees at most $d$.
For a general set $T \subseteq \re^n$,
$int(T)$ denotes its interior, and $\pt T$ denotes its boundary in standard Euclidean topology.
For a symmetric matrix $X$, $X\succeq 0$ (resp., $X\succ 0$) means
$X$ is positive semidefinite (resp. positive definite).
For $u\in \re^N$, $\| u \|_2$ denotes the standard Euclidean norm.
%

\section{Construction of the exact Jacobian SDP relaxation}  \label{sec:constr}
\setcounter{equation}{0}

Let $S$ be the feasible set of \reff{pop:gen} and
\be \label{def:m}
m \quad = \quad \min(m_1+m_2,n-1).
\ee
For convenience, we denote $h(x) = (h_1(x),\ldots,h_{m_1}(x))$
and $g(x) = (g_1(x),\ldots,g_{m_2}(x))$.
For a subset $J=\{j_1,\ldots,j_k\} \subset [m_2]$, dennote
\[
g_J(x) =(g_{j_1}(x), \ldots, g_{j_k}(x)).
\]

Let $x^*$ be a minimizer of \reff{pop:gen}.
If $J=\{j_1,\ldots,j_k\}$ is the index set of $g_j(x^*)=0$
and the KKT conditions hold at $x^*$,
then there exist $\lmd_i$ and $\mu_j (j \in J)$ such that
\begin{align*}
h(x^*) = 0, \quad g_J(x^*)=0, \quad
\nabla f(x^*) = \sum_{i \in [m_1] } \lmd_i \nabla h_i(x^*) +
\sum_{j\in J} \mu_j \nabla g_j(x^*).
\end{align*}
The above implies the Jacobian matrix of $(f,h,g_J)$ is singular at $x^*$.
For a subset $J \subset [m_2]$,
denote the determinantal variety of $(f,h,g_J)$'s Jacobian being singular by
\be \label{def:G_J}
G_J = \left\{ x \in \cpx^n: \rank\, B^J(x) \leq m_1+ |J| \right\}, \quad
\,  B^J(x) = \bbm \nabla f(x) & \nabla h(x) & \nabla g_J(x) \ebm.
\ee
Then, $x^* \, \in  \, V(h,g_J) \cap  G_J$ where  
$
V(h,g_J) := \big \{ x \in \cpx^n: h(x) = 0, g_J(x) = 0 \big \}.
$

This motivates us to use $g_J(x)=0$ and $G_J$ to get tighter SDP relaxations
for \reff{pop:gen}. To do so, a practical issue is
how to get a ``nice" description for $G_J$?
An obvious description for $G_J$ is that all its maximal minors vanish.
But there are totally $\binom{n}{m_1+k+1}$ such minors (if $m_1+k+1\leq n$),
which is huge for big $n,m_1,k$.
Can we define $G_J$ by a set of the smallest number of equations?
Furthermore, the active index set $J$ is usually unknown in advance.
Can we get an SDP relaxation that is independent of $J$?
These issues will be discussed in the sequel.

\subsection{Minimum defining equations for determinantal varieties}
\label{ssec:min-def}

Let $k\leq n$ and $X=(X_{ij})$ be a $n\times k$ matrix of indeterminants $X_{ij}$.
Define the determinantal variety
\[
D_{t-1}^{n,k} = \left\{ X \in \cpx^{n\times k}:\, \rank \, X < t \right\}.
\]
For any index set $I = \{i_1,\ldots, i_k\} \subset [n]$, denote by $\det_I(X)$
the $(i_1,\ldots,i_k)\times (1,\ldots,k)$-minor of matrix $X$, i.e.,
the determinant of the submatrix of $X$ whose row indices are
$i_1,\ldots, i_k$ and column indices are $1,\ldots,k$.
Clearly, it holds that
\[
D_{k-1}^{n,k} = \left\{ X \in \cpx^{n\times k}:\,
\mbox{det}_I(X) = 0 \, \quad \forall \, I \in [n]_k \right\}.
\]
The above has $\binom{n}{k}$ defining equations of degree $k$.
An interesting fact is that
we do not need $\binom{n}{k}$ equations to define $D_{k-1}^{n,k}$.
Actually $nk-k^2+1$ are enough.
There is very nice work on this issue.
Bruns and Vetter \cite{BrVe} showed $nk-t^2+1$
equations are enough for defining $D_{t-1}^{n,k}$.
Later, Bruns and Schw\"{a}nzl \cite{BrSch} showed
$nk-t^2+1$ is the smallest number of equations for defining $D_{t-1}^{n,k}$.
Typically, $nk-t^2+1 \ll \binom{n}{k}$ for big $n$ and $k$.
A general method for constructing $nk-t^2+1$ defining polynomial equations
for $D_{t-1}^{n,k}$ was described in Chapt.~5 of \cite{BrVe}.
Here we briefly show how it works for $D_{k-1}^{n,k}$.

Let $\Gamma(X)$ denote the set of all $k$-minors of $X$
(assume their row indices are strictly increasing).
For convenience, for any $1\leq i_1 < \cdots < i_k\leq n$,
we just denote by $[i_1,\ldots, i_k]$ the $(i_1,\ldots,i_k)\times (1,\ldots,k)$-minor of $X$.
Define a partial ordering on $\Gamma(X)$ as follows:
\[
[i_1,\ldots, i_k] < [j_1,\ldots, j_k] \quad \Longleftrightarrow \quad
i_1 \leq j_1, \ldots, i_k \leq j_k, \,  \sum_{\ell=1}^k i_\ell < \sum_{\ell=1}^k j_\ell.
\]
If $I=\{i_1,\ldots, i_k\}$,
we also write $I= [i_1,\ldots, i_k]$ as a minor in $\Gamma(X)$ for convenience.
For any $I \in \Gamma(X)$, define its rank as
\[
rk(I) = \max \left\{ \ell:\, I = I^{(\ell)} > \cdots > I^{(1)},
\quad \mbox{ every } \quad I^{(i)} \in \Gamma(X)  \right\}.
\]
The maximum minor in $\Gamma(X)$ is $[n-k+1,\ldots,n]$ and has rank $nk-k^2+1$.
For every $1\leq \ell \leq nk-k^2+1$, define
\be \label{def:eta_ell}
\eta_\ell(X) \quad = \quad \sum_{ I \in [n]_k,  rk(I) = \ell }  \mbox{det}_I(X).
\ee

\begin{lem}[Lemma~(5.9), Bruns and Vetter \cite{BrVe}]
It holds that
\[
D_{k-1}^{n,k} = \left\{ X \in \cpx^{n\times k}:\, \eta_\ell(X) = 0,\, \ell =1,\ldots, nk-k^2+1 \right\}.
\]
\end{lem}

When $k=2$, $D_1^{n,2}$ would be defined by $2n-3$ polynomials.
The biggest minor is $[n-1,n]$ and has rank $2n-3$.
For each $\ell =1,2,\ldots, 2n-3$, we clearly have
\[
\eta_\ell(X) = \sum_{1\leq i_1 < i_2 \leq n: \, i_1+i_2 = \ell+2 }  [i_1,i_2].
\]
Every $2$-minor of $X$ is a summand of some $\eta_\ell(X)$.

When $k=3$, $D_2^{n,3}$ can be defined by $3n-8$ polynomials of the form $\eta_\ell(X)$.
For instance of $n=6$, the partial ordering on $\Gamma(X)$ is shown in the following diagram.
{\small
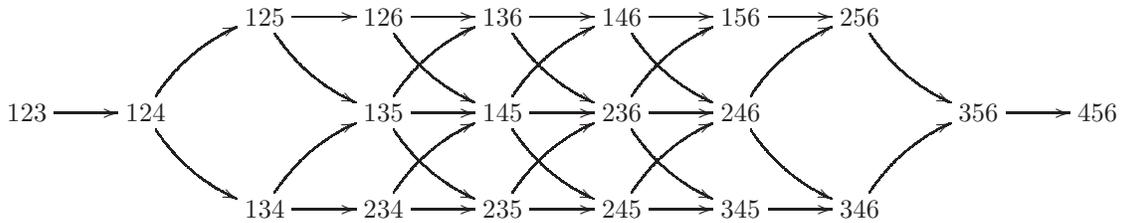
\[
\xymatrix{
     &     & 125  \ar@//[r] \ar@/_/[dr]  & 126 \ar@//[r] \ar@/_/[dr]
     & 136 \ar@//[r] \ar@/_/[dr]  & 146 \ar@//[r] \ar@/_/[dr]
     & 156  \ar@//[r] & 256  \ar@/_/[dr]  &     &   \\
123 \ar@//[r] & 124 \ar@/^/[ur] \ar@/_/[dr] &     & 135  \ar@/^/[ur]  \ar@//[r] \ar@/_/[dr]
    & 145 \ar@/^/[ur] \ar@//[r] \ar@/_/[dr]  & 236 \ar@/^/[ur]  \ar@//[r] \ar@/_/[dr]
    & 246  \ar@/^/[ur] \ar@/_/[dr] &     & 356 \ar@//[r] & 456  \\
     &     & 134 \ar@//[r] \ar@/^/[ur]  & 234 \ar@//[r] \ar@/^/[ur]
     & 235 \ar@//[r] \ar@/^/[ur] & 245 \ar@//[r] \ar@/^/[ur]
     & 345 \ar@//[r] & 346 \ar@/^/[ur] &     &   \\
}
\]
}
In the above an arrow points to a bigger minor. Clearly, we have the expressions
\[
\eta_1(X) = [1,2,3], \quad \eta_2(X) = [1,2,4], \quad \eta_3(X) = [1,2,5]+[1,3,4],
\]
\[
\eta_4(X) = [1,2,6]+[1,3,5]+[2,3,4], \quad \eta_5(X) = [1,3,6]+[1,4,5]+[2,3,5],
\]
\[
\eta_6(X) = [1,4,6]+[2,3,6]+[2,4,5], \quad \eta_7(X) = [1,5,6]+[2,4,6]+[3,4,5],
\]
\[
\eta_8(X)=[2,5,6]+[3,4,6], \quad \eta_9(X)=[3,5,6], \quad \eta_{10}(X)=[4,5,6].
\]
Every above $\eta_i(X)$ has degree $3$.
Note that the summands $[i_1,i_2,i_3]$ from the same $\eta_i(X)$
have a constant summation $i_1+i_2+i_3$.
Thus, for each $\ell =1,\ldots, 3n-8$, we have
\[
\eta_\ell(X) = \sum_{1\leq i_1 < i_2 < i_3 \leq n: \, i_1+i_2+i_3 = \ell+ 5 }  [i_1,i_2,i_3].
\]

When $k>3$ is general, $D_{k-1}^{n,k}$ can be defined by $nk-k^2+1$ polynomials of the form $\eta_\ell(X)$.
For each $\ell =1,2,\ldots, nk-k^2+1$, we similarly have the expression
\[
\eta_\ell(X) = \sum_{1\leq i_1 < \cdots < i_k \leq n: \, i_1+\cdots+i_k = \ell+\binom{k+1}{2}-1 }
[i_1,\ldots,i_k].
\]

\subsection{The exact Jacobian SDP relaxation}

For every $J = \{j_1,\ldots,j_k\} \subset [m_2]$ with $k \leq m-m_1$,
by applying formula \reff{def:eta_ell}, let
\[
\eta_1^J,\ldots, \eta_{len(J)}^J \quad \mbox{ where }
\quad len(J)=n(m_1+k+1)-(m_1+k+1)^2+1
\]
be the set of defining polynomials for
the determinantal variety $G_J$ defined in \reff{def:G_J}
of the Jacobian of $(f,h,g_J)$ being singular.
For each $i=1,\ldots, len(J)$, define
\be  \label{df:vphi-J}
\varphi_i^J(x) = \eta_i^J(x) \cdot \prod_{j \in J^c} g_j(x),
\quad \mbox{ where } \quad J^c=[m_2]\backslash J.
\ee
Using the product $\prod_{j \in J^c} g_j(x)$ in the above
is motivated by a characterization of critical points in \cite{Hiep10}.
For simplicity, list all possible polynomials $\varphi_i^J(x)$ in \reff{df:vphi-J} sequentially as
\be \label{def:vphi:r}
\varphi_1(x), \, \varphi_2(x), \, \ldots, \, \varphi_r(x), \quad \mbox{ where } \quad
r = \sum_{J \subset [m_2], |J|\leq m-m_1}  len(J).
\ee
Now wefine the variety
\be \label{def:var-W}
W = \left\{ x \in \cpx^n: \,  h_1(x) =  \cdots = h_{m_1}(x) =
\varphi_1(x) = \cdots = \varphi_r(x)=0
\right\}.
\ee
If the minimum $f_{min}$ of \reff{pop:gen} is achieved at a KKT point,
then \reff{pop:gen} is equivalent to
\be  \label{pop:kktvar}
\baray{rl}
\underset{x\in \re^n}{\min} & f(x) \\
\mbox{s.t.} & h_1(x) = \cdots = h_{m_1}(x) = 0, \\
& \varphi_1(x) = \cdots = \varphi_r(x) = 0, \\
& g_\nu(x) \geq 0, \,\forall \nu \in \{0,1\}^{m_2}.
\earay
\ee
Here, we denote $g_\nu(x) =g_1(x)^{\nu_1} \cdots g_{m_2}(x)^{\nu_{m_2}}$.

To construct an SDP relaxation for \reff{pop:kktvar},
we need to define localizing moment matrices.
Let $q(x)$ be a polynomial with $\deg(q) \leq 2N$.
Define symmetric matrices $A_\af^{(N)}$ such that
\[
q(x) [x]_d [x]_d^T \quad = \quad
\sum_{\af \in \N^n: |\af| \leq 2N}  A_\af^{(N)} x^\af,
\quad \mbox{ where } \quad d = N-\lceil \deg(q)/2 \rceil.
\]
Then the $N$-th order localizing moment matrix of $q$ is defined as
\[
L_q^{(N)}(y) = \sum_{\af \in \N^n: |\af| \leq 2N}  A_\af^{(N)} y_\af.
\]
Here $y$ is a moment vector indexed by $\af \in \N^n$ with $|\af| \leq 2N$.
Moreover, denote
\[
L_f(y) = \sum_{\af \in \N^n: |\af| \leq \deg(f) }  f_\af y_\af
\quad \mbox{ for } \quad  f(x) = \sum_{\af \in \N^n: |\af| \leq \deg(f) }  f_\af x^\af.
\]

The $N$-th order Lasserre's relaxation for \reff{pop:kktvar} is the SDP
\be  \label{las-sdp:deg=2N}
\baray{rl}
f_N^{(1)}:= \min  &  L_f(y)  \\
\mbox{s.t.} & L_{h_i}^{(N)}(y)= 0, \, i = 1,\ldots,m_1, \\
& L_{\varphi_j}^{(N)}(y)= 0, \, j = 1,\ldots,r, \\
& L_{g_\nu}^{(N)}(y) \succeq 0, \,\forall \nu \in \{0,1\}^{m_2}, \, y_0 = 1.
\earay
\ee
Compared to Schm\"{u}dgen type Lasserre's relaxation, by \reff{def:vphi:r},
the number of new equations in \reff{las-sdp:deg=2N} is
$r = O\Big(2^{m_2}\cdot n \cdot (m_1+m_2)\Big)$.
That is, $r$ is of linear order in $nm_1$ for fixed $m_2$,
but is exponential in $m_2$.
So, when $m_2$ is small or moderately large, \reff{las-sdp:deg=2N} is practical;
but for big $m_2$, \reff{las-sdp:deg=2N} becomes more difficult to solve numerically.
Now we present the dual of \reff{las-sdp:deg=2N}.
Define the truncated preordering $P^{(N)}$ generated by $g_j$ as
\be \label{def:P^(N)}
P^{(N)} = \left\{ \sum_{\nu \in \{0,1\}^{m_2}} \sig_\nu(x) g_\nu(x):
\deg(\sig_\nu g_\nu) \leq 2N \right\},
\ee
and the truncated ideal $I^{(N)}$ generated by $h_i$ and $\varphi_j$ as
\be \label{def:I^(N)}
I^{(N)} = \left\{
\sum_{i=1}^{m_1} p_i(x) h_i(x) + \sum_{j=1}^{r} q_j(x) \varphi_j(x):
\baray{c}
\deg(p_ih_i) \leq 2N \quad \forall\, i\\
deg(q_j\varphi_j) \leq 2N \quad \forall\,j
\earay
\right\}.
\ee
Then, as shown in Lasserre \cite{Las01},
the dual of \reff{las-sdp:deg=2N} is the following SOS relaxation for \reff{pop:kktvar}:
\be \label{sos:deg=2N}
\baray{rl}
f_N^{(2)}:= \max  & \gamma   \\
\mbox{s.t.} & f(x) - \gamma \in I^{(N)}+ P^{(N)}.
\earay
\ee
Note the relaxation \reff{sos:deg=2N} is stronger than \reff{sos:Smg}.
Let $f^*$ be the optimal value of \reff{pop:kktvar}.
Then, by weak duality, we have the relation
\be  \label{ineq:w-dual}
 f_N^{(2)} \, \leq \, f_N^{(1)} \, \leq  \, f^*.
\ee

We are going to show that when $N$ is big enough,
\reff{las-sdp:deg=2N} is an exact SDP relaxation for \reff{pop:kktvar},
i.e., $f_N^{(2)} = f_N^{(1)} = f^*$.
For this purpose, we need the following assumption.

\begin{ass}  \label{as:fin+smth}
(i) $m_1\leq n$.
(ii) For any $u\in S$, at most $n-m_1$ of $g_1(u), \ldots, g_{m_2}(u)$ vanish.
(iii) For every $J=\{j_1,\ldots,j_k\} \subset [m_2]$ with $k\leq n-m_1$,
the variety $V(h,g_J) = \left\{ x\in \cpx^n: h(x)=0,\, g_J(x) =0 \right\}$
is nonsingular (its Jacobian has full rank on $V(h,g_J)$).
\end{ass}

\begin{theorem} \label{thm:sos-exact}
Suppose Assumption~\ref{as:fin+smth} holds.
Let $f^*$ be the minimum of \reff{pop:kktvar}.
Then there exists an integer $N^*>0$ such that
$f_{N}^{(1)} = f_{N}^{(2)} = f^*$ for all $N\geq N^*$.
Furthermore, if the minimum $f_{min}$ of \reff{pop:gen} is achievable, then
$f_{N}^{(1)} = f_{N}^{(2)} = f_{min}$ for $N\geq N^*$.
\end{theorem}

Theorem \ref{thm:sos-exact}
will be proved in Section~\ref{sec:proof}.
When the feasible set $S$ of \reff{pop:gen} is compact,
the minimum $f_{min}$ is always achievable.
Thus, Theorem~\ref{thm:sos-exact} implies the following.

\begin{cor}
Suppose Assumption~\ref{as:fin+smth} holds. If $S$ is compact,
then $f_{N}^{(1)} = f_{N}^{(2)} = f_{min}$ for $N$ big enough.
\end{cor}

A practical issue in applications is how to identify
whether \reff{las-sdp:deg=2N} is exact for a given $N$.
This would be possible by applying the flat-extension condition (FEC) \cite{CuFi00}.
Let $y^*$ be a minimizer of \reff{las-sdp:deg=2N}.
We say $y^*$ satisfies FEC if
$\rank\,L_{g_\nu}^{(N)}(y^*) = \rank\,L_{g_\nu}^{(N-1)}(y^*)$ for every $\nu$.
When FEC holds, \reff{las-sdp:deg=2N} is exact for \reff{pop:gen},
and a finite set of global minimizers would be extracted from $y^*$.
We refer to \cite{HenLas05} for a numerical method on how to do this.
A very nice software for solving SDP relaxations from polynomial optimization
is {\it GloptiPoly~3} \cite{GloPol3}
which also provides routines for finding minimizers if FEC holds.

\bigskip
Now we discuss how general the conditions of Theorem~\ref{thm:sos-exact} are.
Define
\[
B_d(S) = \left\{ f\in \re[x]_{\leq d}:\, \inf_{u\in S} \, f(u) > -\infty \right\}.
\]
Clearly, $B_d(S)$ is convex and has nonempty interior.
Define the projectivization of $S$ as
\be \label{def:Shom}
S^{prj} = \left\{ \tilde{x} \in \re^{n+1}:\,
\tilde{h}_1(\tilde{x})=\cdots =\tilde{h}_{m_1}(\tilde{x})=0, \quad
\tilde{g}_1(\tilde{x})\geq 0, \ldots, \tilde{g}_{m_2}(\tilde{x})\geq 0
\right\}.
\ee
Here $\tilde{p}$ denotes the homogenization of $p$,
and $\tilde{x}=(x_0,x_1,\ldots,x_n)$, i.e.,
$
\tilde{p}(\tilde{x}) \, = \, x_0^{\deg(p)}p(x/x_0).
$
We say $S$ is closed at $\infty$ if
\[
S^{prj} \cap \{x_0\geq 0\} \, = \,
\mbox{closure}\left( S^{prj} \cap \{x_0 > 0\} \right).
\]
Under some generic conditions,
Assumption~\ref{as:fin+smth} is true
and the minimum $f_{min}$ of \reff{pop:gen} is achievable.
These conditions are expressed as non-vanishing
of the so-called {\it resultants $Res$} or {\it discriminants $\Delta$},
which are polynomial in the coefficients of $f,h_i,g_j$.
We refer to Appendix for a short introduction about $Res$ and $\Delta$.

\begin{theorem} \label{thm:con-gen}
Let $f,h_i,g_j$ be the polynomials in \reff{pop:gen},
and $S$ be the feasible set.
\bit
\item [(a)] If $m_1 > n$ and
$Res(h_{i_1}, \ldots, h_{i_{n+1}}) \ne 0$
for some $\{i_1,\ldots,i_{n+1} \}$, then $S=\emptyset$.

\item [(b)] If $m_1 \leq n$ and for every $\{j_1,\ldots,j_{n-m_1+1} \} \subset [m_2]$
\[
Res(h_1,\ldots, h_{m_1}, g_{j_1}, \ldots, g_{j_{n-m_1+1}}) \, \ne \, 0,
\]
then item (ii) of Assumption~\ref{as:fin+smth} holds.

\item [(c)] If $m_1 \leq n$ and for every $\{j_1,\ldots,j_k\} \subset [m_2]$ with $k\leq n-m_1$
\[
\Delta(h_1,\ldots, h_{m_1}, g_{j_1}, \ldots, g_{j_k}) \, \ne \, 0,
\]
then item (iii) of Assumption~\ref{as:fin+smth} holds.

\item [(d)] Suppose $S$ is closed at $\infty$ and $f\in B_d(S)$.
If the resultant of any $n$ of  $h_i^{hom}$, $g_j^{hom}$ is nonzero
(only when $m_1+m_2\geq n$),
and for every $\{j_1,\ldots,j_k\}$ with $k\leq n-m_1-1$
\[
\Delta(f^{hom}, h_1^{hom},\ldots, h_{m_1}^{hom},
g_{j_1}^{hom}, \ldots, g_{j_k}^{hom} ) \, \ne \, 0,
\]
then there exists $v\in S$ such that $f_{min} = f(v)$.
Here $p^{hom}$ denotes $p$'s homogeneous part of the highest degree.

\item [(e)] If $f\in B_d(\re^n)$ and $\Delta(f^{hom}) \ne 0$,
then the minimum of $f(x)$ in $\re^n$ is achievable.

\eit
\end{theorem}

Theorem~\ref{thm:con-gen} will be proved in Section~\ref{sec:proof}.
Now we consider the special case of \reff{pop:gen}
having no constraints. If $f_{min} > -\infty$ is achievable,
then \reff{pop:gen} is equivalent to \reff{po:grad=0}.
The item (e) of Theorem~\ref{thm:con-gen} tells us that
this is generically true.
The gradient SOS relaxation for \reff{po:grad=0} described in \cite{NDS}
is the same as \reff{sos:deg=2N} for the unconstrained case.
The following is an immediate consequence of Theorem~\ref{thm:sos-exact}
and item (e) of Theorem~\ref{thm:con-gen}.

\begin{cor} \label{cor:gsos}
If $S=\re^n$, $f(x)$ has minimum $f_{min}>-\infty$, and $\Delta(f^{hom}) \ne 0$,
then the optimal values of \reff{las-sdp:deg=2N} and \reff{sos:deg=2N}
are equal to $f_{min}$ if $N$ is big enough.
\end{cor}

Corollary~\ref{cor:gsos} is stronger than Theorem~10 of \cite{NDS},
where the exactness of gradient SOS relaxation for a finite order $N$
is only shown when the gradient ideal is radical.

\section{Proof of exactness and genericity}  \label{sec:proof}
\setcounter{equation}{0}

This section proves Theorems \ref{thm:sos-exact} and \ref{thm:con-gen}.
First, we give some lemmas that are crucially used in the proof.

\begin{lem} \label{lm:W=Kx}
Let $K$ be the variety defined by the KKT conditions
\be \label{def:kkt-K}
K =
\left\{ (x,\lmd,\mu) \in \cpx^{n +m_1 + m_2}:\,
\baray{c}
\nabla f(x) = \overset{m_1}{ \underset{i=1}{\sum} } \lmd_i \nabla h_i(x)
+ \overset{m_2}{ \underset{j=1}{\sum} }   \mu_j \nabla g_j(x) \\
h_i(x) = \mu_j g_j(x) = 0, \, \forall \, (i,j) \in [m_1] \times [m_2]
\earay
\right\}.
\ee
If Assumption~\ref{as:fin+smth} holds, then $W=K_x$ where
\[
K_x=\{x \in \cpx^n:\, (x,\lmd,\mu) \in K \mbox{ for some } \lmd, \mu\}.
\]
\end{lem}
\begin{proof}
First, we prove $W\subset K_x$.  Choose an arbitrary $u\in W$.
Let $J=\{ j\in [m_2]: \, g_j(u) = 0\}$ and $k=|J|$.
By Assumption~\ref{as:fin+smth}, $m_1+k\leq n$.
Recall from \reff{def:G_J} that 
\[
B^J(x) \,= \, \bbm  \nabla f(x) & \nabla h(x) & \nabla g_J(x) \ebm.
\]

\noindent {\bf Case $m_1+k = n$} \,
By Assumption~\ref{as:fin+smth}, the matrix
$H(u) = \bbm \nabla h(u) & \nabla g_J(u) \ebm $ is nonsingular.
Note that $H(u)$ is now a square matrix. So, $H(u)$
is invertible, and there exist $\lmd_i$ and $\mu_j (j \in J)$ such that
\be \label{Df=Dh+Dg}
\nabla f(u) =  \sum_{i=1}^{m_1} \lmd_i \nabla h_i(u) + \sum_{j\in J} \mu_j \nabla g_j(u)
\ee
Define $\mu_j = 0$ for $j\not\in J$, then we have $u\in K_x$.

\smallskip
\noindent {\bf Case $m_1+k < n$} \,
By the construction of polynomials $\varphi_i(x)$ in \reff{def:vphi:r},
some of them are
\[
\varphi_i^J(x) \, := \,
\eta_i(B^J(x)) \cdot \prod_{j\in J^c} g_j(x),
\quad i =1,\ldots, n(m_1+k+1)-(m_1+k+1)^2+1.
\]
So the equations $\varphi_i(u)=0$ imply every above $\varphi_i^J(u)=0$
(see its definition in \reff{df:vphi-J}).
Hence $B^J(u)$ is singular.
By Assumption~\ref{as:fin+smth}, the matrix $H(u)$ is nonsingular.
So there exist $\lmd_i$ and $\mu_j (j \in J)$ satisfying \reff{Df=Dh+Dg}.
Define $\mu_j = 0$ for $j\not\in J$, then we also have $u\in K_x$.

\medskip
Second, we prove $K_x \subset W$.  Choose an arbitrary $u\in K_x$ with $(u,\lmd,\mu) \in K$.
Let $I=\{ j \in [m_2]: \, g_j(u) = 0\}$. If $I=\emptyset$, then $\mu=0$,
and $\bbm \nabla f(u) & \nabla h(u)\ebm$ and $B^J(u)$ are both singular,
which implies all $\varphi_i(u)=0$ and $u\in W$.
If $I\ne \emptyset$, write $I=\{i_1,\ldots, i_t\}$.
Let $J=\{j_1,\ldots, j_k\} \subset [m_2]$ be an arbitrary index set
with $m_1+k \leq m$.

\noindent {\bf Case $I \nsubseteq J$} \,
At least one $j\in J^c$ belongs to $I$.
By choice of $I$, we know from \reff{df:vphi-J}
\[
\varphi_i^J(u) = \eta_i(B^J(u)) \cdot \prod_{j\in J^c} g_j(u) = 0.
\]

\smallskip
\noindent {\bf Case $I \subseteq J$}\, Then $\mu_j=0$ for all $j \in J^c$.
By definition of $K$, the matrix $B^J(u)$
must be singular. All polynomials $\varphi_i^J(x)$ vanish at $u$ by their construction.

Combining the above two cases, we know all $\varphi_i^J(x)$ vanish at $u$,
that is, $\varphi_1(u)=\cdots = \varphi_r(u)=0$.
So $u \in W$.
\end{proof}

\begin{lem} \label{lm:W-dcmp}
Suppose Assumption~\ref{as:fin+smth} is true.
Let $W$ be defined in \reff{def:var-W},
and $T=\{x \in \re^n:g_j(x)\geq 0, j =1,\ldots,m_2\}$.
Then there exist disjoint subvarieties $W_0,W_1, \ldots, W_r$ of $W$
and distinct $v_1,\ldots, v_r \in \re$ such that
\[
W=W_0 \cup W_1 \cup \cdots \cup W_r, \quad
 W_0 \cap T = \emptyset,  \quad
 W_i \cap T \ne \emptyset, i=1,\ldots,r,
\]
and $f(x)$ is constantly equal to $v_i$ on $W_i$ for $i=1,\ldots,r$.
\end{lem}
\begin{proof}
Let $K= K_1 \cup \cdots \cup K_r$ be a decomposition of irreducible varieties.
Then $f(x)$ is equaling a constant $v_i$ on each $K_i$, as shown by Lemma~3.3 in \cite{DNP}.
By grouping all $K_i$ for which $v_i$ are same into a single variety,
we can assume all $v_i$ are distinct.
Let $\widehat{W}_i$ be the projection of $K_i$ into $x$-space,
then by Lemma~\ref{lm:W=Kx} we get
\[
W = \widehat{W}_1 \cup \cdots \cup \widehat{W}_r.
\]
Let $W_i = Zar(\widehat{W}_i)$.
Applying Zariski closure in the above gives
\[
W = Zar(W) =  W_1 \cup \cdots \cup W_r.
\]
Note that $f(x)$ still achieves a constant value on each $W_i$.
Group all $W_j$ for which $W_j \cap T =\emptyset$ into a single variety $W_0$
(if every $W_j \cap T \ne \emptyset$ we set $W_0 = \emptyset$).
For convenience, we still write the resulting decomposition as
$W=W_0 \cup W_1 \cup \cdots \cup W_r$.
Clearly, $W_0 \cap T = \emptyset$,
and for $i>0$ the values $v_i$ are real and distinct
(because $\emptyset \ne W_i \cap T \subset \re^n$ and $f(x)$ has real coefficients).
Since $f(x)$ achieves distinct values on different $W_i$,
we know $W_i$ must be disjoint from each other.
Therefore, we get a desired decomposition for $W$.
\end{proof}

\begin{lem} \label{lm:unit-dcmp}
Let $I_0,I_1,\ldots,I_k$ be ideals of $\re[x]$ such that
$V(I_i) \cap V(I_j) = \emptyset$ for distinct $i,j$,
and $I = I_0 \cap I_1 \cap \cdots \cap I_k$.
Then there exist $a_0,a_1,\ldots, a_k \in \re[x]$ satisfying
\[
a_0^2+\cdots+a_k^2 -1 \, \in \, I, \quad  a_i \, \in \,
\bigcap_{ i \ne j \in \{0,\ldots,k\} } I_j.
\]
\end{lem}
\begin{proof}
We prove by induction.
When $k=1$, by Theorem~\ref{Null-Weak}, there exist
$p \in I_0, q\in I_1$ such that $p+q=1$.
Then $a_0=p,a_1=q$ satisfy the lemma.

Suppose the lemma is true for $k=t$. We prove it is also true for $k=t+1$.
Let $J= I_0 \cap \cdots \cap I_t$. By induction,
there exist $b_0,\ldots,b_t \in \re[x]$ such that
\[
b_0^2+\cdots+b_t^2 -1 \, \in \, J, \quad
b_i \, \in \, \bigcap_{ i \ne j \in \{0,\ldots,t\} } I_j, \quad i=0,\ldots,t.
\]
Since $V(I_{t+1})$ is disjoint from $V(J) = V(I_0) \cup \cdots \cup V(I_t)$,
by Theorem~\ref{Null-Weak},
there exist $p \in I_{t+1}$ and $q\in J$ such that $p+q = 1$.
Let $a_i = b_i p$ for $i=0,\ldots,t$ and $a_{t+1} = q$. Then
\[
a_i \in  \bigcap_{i \ne j \in \{0,\ldots,t+1\} } I_j, \quad i = 0,\ldots,t+1.
\]
Since $(p+q)^2=1, I =I_{t+1}\cap J$, we have $pq \in I$,
$(b_0^2+\cdots+b_t^2-1)p^2 \in I$, and
\begin{align*}
a_0^2+a_1^2+\cdots+a_{t+1}^2 - 1
= (b_0^2+\cdots+b_t^2-1)p^2 - 2pq \in I,
\end{align*}
which completes the proof.
\end{proof}

\begin{theorem} \label{pro:unf-deg}
Suppose Assumption~\ref{as:fin+smth} is true and
let $f^*$ be the minimum of \reff{pop:kktvar}.
Then there exists an integer $N^*>0$ such that for every $\eps>0$
\be
f(x) - f^* + \eps \, \in \, I^{(N^*)} + P^{(N^*)}.
\ee
\end{theorem}
\begin{proof}
Generally, we can assume $f^*=0$.
Decompose $W$ as in Lemma~\ref{lm:W-dcmp}. Then
\[
f^* \quad = \quad \min \{v_1,\ldots, v_r\}.
\]
Reorder $W_i$ such that $v_1 > v_2 > \cdots > v_r = 0.$
The ideal
\be  \label{df:I-W}
I_W \quad = \quad \langle h_1, \ldots, h_{m_1}, \varphi_1, \ldots, \varphi_r \rangle
\ee
has a primary decomposition (see Sturmfels \cite[Chapter~5]{Stu02})
\[
I_W = E_0 \,\cap \, E_1 \,\cap \,\cdots \,\cap \, E_r
\]
such that each ideal $E_i\subset \re[x]$ has variety $W_i = V(E_i)$.

\smallskip
When $i=0$, we have $V_{\re}(E_0) \cap T = \emptyset$
($T$ is defined in Lemma~\ref{lm:W-dcmp}).
By Theorem~\ref{Pos-Nul}, there exist SOS polynomials $\tau_\nu$ satisfying
\[
-1 \equiv \sum_{\nu \in \{0,1\}^{m_2} } \tau_\nu \cdot
g_\nu(x)  \quad mod \quad E_0 .
\]
Thus, from $f = \frac{1}{4} (f+1)^2 - \frac{1}{4} (f-1)^2$, we have
\begin{align*}
f & \equiv \frac{1}{4} \left\{
(f+1)^2 + (f-1)^2 \sum_{\nu \in \{0,1\}^{m_2} }
\tau_\nu \cdot g_\nu  \right\}
\quad mod \quad E_0   \\
& \equiv \sum_{\nu \in \{0,1\}^{m_2} }
\widehat{\tau_\nu}  \cdot g_\nu
\quad mod \quad E_0
\end{align*}
for certain SOS polynomials $\widehat{\tau_\nu}$. Let
\[
\sig_0 = \eps +
\sum_{\nu \in \{0,1\}^{m_2} } \widehat{\tau_\nu} \cdot g_\nu.
\]
Clearly, if $N_0>0$ is big enough, then
$\sig_0 \in P^{(N_0)}$ for all $\eps>0$. Let
$q_0 = f+\eps - \sig_0 \, \in \, E_0$,
which is independent of $\eps$.

\smallskip
For each $i=1,\ldots, r-1$, $v_i>0$ and $v_i^{-1} f(x) - 1$ vanishes on $W_i$.
By Theorem~\ref{Null-Strong}, there exists $k_i > 0$ such that
$
( v_i^{-1} f(x) - 1)^{k_i} \in E_i.
$
Thus, it holds that
\[
s_i(x) := \sqrt{v_i} \Big(1 \, + \, \big( v_i^{-1} f(x) - 1 \big)\Big)^{1/2}
\, \equiv \,  \sqrt{v_i}  \sum_{j=0}^{k_i-1} \binom{1/2}{j} (v_i^{-1} f(x) - 1)^j\,\,\,\,
mod\,\,\,\, E_i\, .
\]
Let $\sig_i =  s_i(x)^2 + \eps$, and
$q_i = f + \eps - \sig_i \, \in \, E_i$,
which is also independent of $\eps > 0$.

\smallskip
When $i=r$, $v_r=0$ and $f(x)$ vanishes on $W_r$.
By Theorem~\ref{Null-Strong}, there exists $k_r > 0$
such that $f(x)^{k_r} \in E_r$.
Thus we obtain that
\[
s_r(x) := \sqrt{\eps} \left(1 + \eps^{-1} f(x) \right)^{1/2}
 \equiv \sqrt{\eps}  \sum_{j=0}^{k_r-1} \binom{1/2}{j} \eps^{-j} f(x)^j\,\,\,\,
mod\,\,\,\, E_r\, .
\]
Let $\sig_r = s_r(x)^2$, and
$q_r = f + \eps - \sig_r  \, \in \,  E_r$.
Clearly, we have
\[
q_r(x) \, = \, \sum_{j=0}^{k_r-2} c_j(\eps) f(x)^{k_r+j}
\]
for some real scalars $c_j(\eps)$. Note each $f(x)^{k_r+j}  \in E_r$.

\bigskip 
Applying Lemma~\ref{lm:unit-dcmp} to ideals $E_0,E_1,\ldots,E_r$, 
we can find $a_0, \ldots, a_r \in \re[x]$ satisfying
\[
a_0^2+\cdots+a_r^2-1 \, \in \, I_W, \quad a_i \in
\bigcap_{i \ne j \in \{0,1,\ldots,r \} } E_j.
\]
Let $\sig = \sig_0 a_0^2+\sig_1 a_1^2 + \cdots + \sig_r a_r^2$, then
\begin{align*}
 f(x) + \eps  - \sig
 = \sum_{i=0}^r (f+\eps-\sig_i) a_i^2 + (f+\eps)(1-a_0^2-\cdots-a_r^2).
\end{align*}
Since $q_i = f+\eps-\sig_i \in E_i$, it holds that
\[
(f+\eps-\sig_i) a_i^2 \, \in \, \bigcap_{j=0}^r E_j = I_W.
\]
For each $0\leq i <r$, $q_i$ is independent of $\eps$.
There exists $N_1 >0$ such that for all $\eps >0$
\[
(f+\eps-\sig_i) a_i^2 \, \in \, I^{(N_1)}, \quad i=0,1,\ldots,r-1.
\]
For $i=r$, $q_r=f+\eps-\sig_r$ depends on $\eps$.
By the choice of $q_r$, it holds that
\[
(f+\eps-\sig_r) a_r^2  =  \sum_{j=0}^{k_r-2} c_j(\eps) f^{k_r+j} a_r^2.
\]
Note each $ f^{k_r+j} a_r^2 \in I_W$, since $f^{k_r+j} \in E_r$.
So, there exists $N_2 >0$ such that for all $\eps >0$
\[
(f+\eps-\sig_r) a_r^2 \, \in \, I^{(N_2)}.
\]
Since $1-a_1^2-\cdots-a_r^2 \in I_W$, there also exists $N_3 >0$ such that for all $\eps >0$
\[
(f+\eps)(1-a_1^2-\cdots-a_r^2) \, \in \, I^{(N_3)}.
\]
Combining the above, we know if $N^*$ is big enough, then
$
f(x) + \eps  - \sig \, \in \, I^{(N^*)}
$
for all $\eps >0$.
From the constructions of $\sig_i$ and $a_i$,
we know their degrees are independent of $\eps$.
So, $\sig  \, \in \, P^{(N^*)}$ for all $\eps >0$ if $N^*$ is big enough,
which completes the proof.
\end{proof}

Theorem~\ref{pro:unf-deg} is a kind of Positivstellensatz of
representing $f(x)-f_{min}+\eps$, which is positive on $S$ for all $\eps>0$,
by the preordering generated by $g_j$ modulo the ideal
$I_W$ in \reff{df:I-W} of variety $W$.
Usually, we can not conclude
$f(x)-f_{min} \in I^{(N^*)}+P^{(N^*)}$ by setting $\eps =0$,
because the coefficients of the representing polynomials
of $f(x)-f_{min}+\eps$ in $I^{(N^*)}+P^{(N^*)}$
go to infinity as $\eps \to 0$ (see $s_r(x)$ in the proof).
It is possible that
$f(x)-f_{min} \not\in I^{(N)}+P^{(N)}$ for every $N>0$.
Such a counterexample is Example~\ref{em-unc:0fail}.
However, Theorem~\ref{pro:unf-deg} shows that the degree bound $N^*$
required for representing $f(x)-f_{min}+\eps$ is independent of $\eps$.
This is a crucial property
justifying the exactness of the SDP relaxation \reff{las-sdp:deg=2N}.
Now we present its proof below.

\medskip

\noindent
{\it Proof of Theorem~\ref{thm:sos-exact}} \quad
By Theorem~\ref{pro:unf-deg}, there exists $N^*$
such that for every $\eps >0$
\[
f(x) - (f^*-\eps) \quad \in \quad I^{(N^*)} + P^{(N^*)}.
\]
Since $f_{N^*}^{(1)}, f_{N^*}^{(2)}$ are the optimal values of
\reff{las-sdp:deg=2N} and \reff{sos:deg=2N} respectively, we know
\[
f^*-\eps \leq  f_{N^*}^{(2)} \leq  f_{N^*}^{(1)} \leq f^*.
\]
Because $\eps >0$ is arbitrary, the above implies
$f_{N^*}^{(1)} =  f_{N^*}^{(2)} = f^*$.
Since the sequence $\{f_{N}^{(2)}\}$ is monotonically increasing and every
$f_{N}^{(2)} \leq f_{N}^{(1)} \leq f^*$ by \reff{ineq:w-dual},
we get $f_{N}^{(1)} =  f_{N}^{(2)} = f^*$ for all $N\geq N^*$.
If the minimum $f_{min}$ of \reff{pop:gen} is achievable,
then there exists $x^*\in S$ such that $f_{min}=f(x^*)$.
By Assumption~\ref{as:fin+smth}, we must have $x^* \in W$.
So $x^*$ is feasible for \reff{pop:kktvar}, and $f^* = f_{min}$.
Thus, we also have $f_{N}^{(1)} =  f_{N}^{(2)} = f_{min}$
for all $N\geq N^*$.
\qed

\bigskip

Last we prove Theorem~\ref{thm:con-gen} by using
the properties of resultants and discriminants
described in Appendix.

\medskip
\noindent
{\it Proof of Theorem~\ref{thm:con-gen} } \,
(a) If $Res(h_{i_1},\ldots,h_{i_{n+1}})\ne 0$, then the polynomial system
\[
h_{i_1}(x)=\cdots=h_{i_{n+1}}(x)=0
\]
does not have complex solution. Hence, $V(h)=\emptyset$ and consequently $S=\emptyset$.

(b) For a contradiction, suppose $n-m_1+1$ of
$g_j$ vanish at $u \in S$, say, $g_{j_1},\ldots, g_{j_{n-m_1+1}}$.
Then the polynomial system
\[
h_1(x)=\cdots=h_{m_1}(x)=g_{j_1}(x)=\cdots=g_{j_{n-m_1+1}}(x)=0
\]
has a solution, which contradicts
$Res(h_1,\ldots,h_{m_1},g_{j_1},\ldots,g_{j_{n-m_1+1}})\ne 0$.

(c) For every $J=\{j_1,\ldots,j_k\} \subset [m_2]$ with $k\leq n-m_1$, if
\[
\Delta(h_1,\ldots, h_{m_1}, g_{j_1},\ldots, g_{j_k}) \ne 0,
\]
then the polynomial system
\[
h_1(x)=\cdots=h_{m_1}(x)=g_{j_1}(x)=\cdots=g_{j_k}(x)=0
\]
has no singular solution, i.e., the variety $V(h,g_J)$ is smooth.

(d) Let $f_0(x)=f(x)-f_{min}$. Then $f_0$ lies on the boundary of the set
\[
P_d(S) \, = \, \Big\{ p\in B_d(S):\, p(x) \geq 0 \, \forall \, x \in S \Big\}.
\]
Since $S$ is closed at $\infty$,
by Prop.~6.1 of \cite{NieDis},  $f_0 \in \pt P_d(S)$ implies
\[
0 = \min_{\tilde{x} \in S^{prj}, \|\tilde{x}\|_2=1, x_0 \geq 0 }  \tilde{f}_0(\tilde{x}).
\]
Let $\tilde{u}=(u_0,u_1,\ldots,u_n) \ne 0$ be a minimizer of the above,
which must exist because the feasible set is compact.
We claim that $u_0 \ne 0$. Otherwise, suppose $u_0=0$.
Then $u=(u_1,\ldots,u_n) \ne 0$ is a minimizer of
\[
\baray{rl}
0 = \min  & f^{hom}(x) \\
s.t. &  h_1^{hom}(x) = \cdots =  h_{m_1}^{hom}(x) = 0, \\
& g_{1}^{hom}(x) \geq 0, \ldots,  g_{m_2}^{hom}(x) \geq 0.
\earay
\]
Let $j_1,\ldots,j_k \in [m_2]$ be the indices of active constraints.
By Fritz-John optimality condition (see Sec. 3.3.5 in \cite{Bsks}),
there exists $(\lmd_0, \lmd_1,\ldots,\lmd_{m_1},
\mu_1,\ldots,\mu_k)\ne 0$ satisfying
\[
\baray{c}
\lmd_0 \nabla f^{hom}(u) +
\overset{m_1}{ \underset{i=1}{\sum} } \lmd_i \nabla h_i^{hom}(u)
+ \cdots + \overset{k}{ \underset{\ell=1}{\sum} }
\mu_\ell \nabla g_{j_\ell}^{hom}(u) = 0, \\
f^{hom}(u) = h_1^{hom}(u) = \cdots = h_{m_1}^{hom}(u)=
g_{j_1}^{hom}(u) = \cdots = g_{j_k}^{hom}(u)=0.
\earay
\]
Thus, the homogeneous polynomial system
\[
f^{hom}(x) = h_1^{hom}(x) = \cdots = h_{m_1}^{hom}(x)=
g_{j_1}^{hom}(x) = \cdots = g_{j_k}^{hom}(x)=0
\]
has a nonzero singular solution.
Since the resultant of any $n$ of $h_i^{hom}, g_j^{hom}$ is nonzero,
we must have $m_1+k \leq n-1$. So the discriminant
\[
\Delta(f^{hom},h_1^{hom},\ldots, h_{m_1}^{hom},
g_{j_1}, \ldots, g_{j_k}^{hom})
\]
is defined and must vanish, which is a contradiction.
So $u_0 \ne 0$. Let $v = u/u_0$,
then $\tilde{u} \in S^{prj}$ implies $v\in S$ and
$
f(v)-f_{min} = u_0^{-d}\tilde{f_0}(\tilde{u}) = 0.
$

Clearly, (e) is true since it is a special case of (d).
\qed

\section{Some variations} \label{sec:var}
\setcounter{equation}{0}

This section presents some variations of
the exact SDP relaxation \reff{las-sdp:deg=2N} and its dual \reff{sos:deg=2N}.

\subsection{A refined version based on all maximal minors}

An SDP relaxation tighter than \reff{las-sdp:deg=2N}
would be obtained by using all the maximal minors to define
the determinantal variety $G_J$ in \reff{def:G_J},
while the number of equations would be significantly larger.
For every $ J = \{j_1,\ldots,j_k\} \subset [m_2]$ with $m_1+k \leq m$, let
\[
\tau_1^J,\ldots, \tau_{\ell}^J
\]
be all the maximal minors of $B^J(x)$ defined in \reff{def:G_J}.
Then define new polynomials
\be \label{eq-df:psi-J}
\psi_i^J(x) = \tau_i^J(x) \cdot \prod_{j \in J^c} g_j(x), \quad  i=1,\ldots,\ell.
\ee
List all such possible $\psi_i^J(x)$ as
\[
\psi_1(x), \, \psi_2(x), \, \ldots, \, \psi_t(x), \quad \mbox{ where } \quad
t \,= \, \sum_{J \subset [m_2], |J|\leq m-m_1 } \binom{n}{|J|+m_1+1}.
\]
Like \reff{pop:kktvar}, we formulate \reff{pop:gen} equivalently as
\be   \label{po:allminor}
\baray{rl}
\underset{x\in \re^n}{\min} & f(x) \\
\mbox{s.t.} & h_i(x) = \psi_j(x) = 0,  \, i\in [m_1], j \in [t], \\
& g_\nu(x) \geq 0, \,\forall \nu \in \{0,1\}^{m_2}.
\earay
\ee
The standard $N$-th order Lasserre's relaxation for the above is
\be  \label{las:allminor}
\baray{rl}
\min  &  L_f(y)  \\
\mbox{s.t.} & L_{h_i}^{(N)}(y)= 0,
L_{\psi_j}^{(N)}(y)= 0,  \, i\in [m_1], j \in [t], \\
& L_{g_\nu}^{(N)}(y) \succeq 0, \,\forall \nu \in \{0,1\}^{m_2}, \, y_0 = 1.
\earay
\ee
Note that every $\varphi_i^J$ in \reff{df:vphi-J}
is a sum of polynomials like $\psi_i^J(x)$ in \reff{eq-df:psi-J}.
So the equations $L_{\psi_j}^{(N)}(y)= 0$ in \reff{las:allminor}
implies $L_{\varphi_j}^{(N)}(y)= 0$ in \reff{las-sdp:deg=2N}.
Hence, \reff{las:allminor} is stronger than \reff{las-sdp:deg=2N}.
Its dual is an SOS program like \reff{sos:deg=2N}.
Theorem~\ref{thm:sos-exact} then implies the following.

\begin{cor}
Suppose Assumption~\ref{as:fin+smth} is true,
and the minimum $f_{min}$ of \reff{pop:gen} is achievable.
If $N$ is big enough, then the optimal value of
\reff{las:allminor} is equal to $f_{min}$.
\end{cor}

\subsection{A Lasserre type variation without using cross products of $g_j$}

If the minimum $f_{min}$ of \reff{pop:gen} is achieved at a KKT point,
then \reff{pop:gen} is equivalent to
\be  \label{put:kktvar}
\baray{rl}
\underset{x\in \re^n}{\min} & f(x) \\
\mbox{s.t.} & h_i(x) = \varphi_j(x) = 0, \, i\in [m_1], j \in [r], \\
& g_1(x) \geq 0, \ldots, g_{m_2}(x) \geq 0.
\earay
\ee
The standard $N$-th order Lasserre's relaxation for \reff{put:kktvar} is
\be  \label{put-sdp:N}
\baray{rl}
\min  &  L_f(y)  \\
\mbox{s.t.} & L_{h_i}^{(N)}(y)= 0, L_{\varphi_j}^{(N)}(y)= 0, \, i\in [m_1], j \in [r], \\
& L_{g_i}^{(N)}(y) \succeq 0, \, i=0,1,\ldots, m_2, \, y_0 = 1.
\earay
\ee
The difference between \reff{put-sdp:N} and \reff{las-sdp:deg=2N} is that
the cross products of $g_j(x)$ are not used in \reff{put-sdp:N},
which makes the number of resulting LMIs much smaller.
Similar to $P^{(N)}$, define the truncated quadratic module $M^{(N)}$
generated by $g_i$ as
\be \label{def:M^(N)}
M^{(N)} = \left\{ \sum_{i=0}^{m_2} \sig_i(x) g_i(x):
\deg(\sig_i g_i) \leq 2N \right\}.
\ee
The dual of \reff{put-sdp:N} would be shown to be
the following SOS relaxation for \reff{put:kktvar}:
\be \label{put-sos:N}
\baray{rl}
\max  & \gamma   \\
\mbox{s.t.} & f(x) - \gamma \in I^{(N)}+ M^{(N)}.
\earay
\ee
Clearly, for the same $N$,
\reff{put-sos:N} is stronger than the standard
Lasserre's relaxation \reff{sos:Put}.
To prove \reff{put-sdp:N} and \reff{put-sos:N} are exact for some $N$,
we need the archimedean condition (AC) for $S$, i.e.,
there exist $R>0$, $\phi_1(x),\ldots,\phi_{m_1}(x) \in \re[x]$
and SOS $s_0(x),\ldots, s_{m_2}(x) \in \re[x]$ such that
\[
R-\|x\|_2^2 \, = \, \sum_{i=1}^{m_1} \phi_i(x) h_i(x) +
\sum_{j=0}^{m_2} s_j(x) g_j(x).
\]

\begin{theorem} \label{thm:las-exact}
Suppose Assumption~\ref{as:fin+smth} and the archimedean condition hold.
If $N$ is big enough, then the optimal values of
\reff{put-sdp:N} and \reff{put-sos:N} are equal to $f_{min}$.
\end{theorem}

To prove Theorem~\ref{thm:las-exact}, we need the following.

\begin{theorem} \label{las-M:dg-bd}
Suppose Assumption~\ref{as:fin+smth} and the archimedean condition hold.
Let $f^*$ be the optimal value of \reff{put:kktvar}.
Then there exists an integer $N^*>0$ such that for every $\eps>0$
\be
f(x) - f^* + \eps \, \in \, I^{(N^*)} + M^{(N^*)}.
\ee
\end{theorem}
\begin{proof}
The proof is almost same as for Theorem~\ref{pro:unf-deg}.
We follow the same approach used there.
The only difference occurs for the case $i=0$ and $V_\re(E_0)\cap T = \emptyset$.
By Theorem~\ref{Pos-Nul},
there exist SOS polynomials $\eta_\nu$ satisfying
\[
-2 \equiv \sum_{\nu \in \{0,1\}^{m_2} } \eta_\nu \cdot
g_\nu  \quad mod \quad E_0 .
\]
Clearly, each
$\frac{1}{2^{m_2}} + \eta_\nu \cdot g_1^{\nu_1} \cdots g_{m_2}^{\nu_{m_2}}$ is positive on $S$.
Since AC holds, by Putinar's Positivtellensatz (Theorem~\ref{Put-Pos}),
there exist SOS polynomials $\theta_{\nu,i}$ such that
\[
\frac{1}{2^{m_2}}+\eta_\nu \cdot g_\nu
= \sum_{i=0}^{m_2} \theta_{\nu,i} g_i
\quad \mbox{ mod } \quad \langle h_1,\ldots,h_{m_1} \rangle.
\]
Hence, it holds that
\begin{align*}
-1 & \equiv \sum_{ \nu \in \{0,1\}^{m_2} } \,
\left( \frac{1}{2^{m_2}}+ \eta_\nu \cdot g_\nu \right)
\quad \mbox{ mod } \quad  \langle h_1,\ldots,h_{m_1} \rangle  + E_0 \\
& \equiv  \sum_{i=0}^{m_2}   \,
\left(  \sum_{ \nu \in \{0,1\}^{m_2} } \theta_{\nu,i}  \right) g_i
\quad \mbox{ mod } \quad E_0.
\end{align*}
The second equivalence above is due to the relation
\[
\langle h_1,\ldots,h_{m_1} \rangle
\, \subset \, I_W  \, \subset \, E_0.
\]
Letting $\tau_i =  \sum_{ \nu \in \{0,1\}^{m_2} } \theta_{\nu,i}$,
which is clearly SOS, we get
\[
-1 \equiv \tau_0 + \tau_1 g_1 + \cdots + \tau_{m_2} g_{m_2}
\quad mod \quad E_0.
\]
The rest of the proof is almost same as for Theorem~\ref{pro:unf-deg}.
\end{proof}

\bigskip
\noindent
{\it Proof of Theorem~\ref{thm:las-exact}} \quad
For convenience, still let $f_{N}^{(1)}, f_{N}^{(2)}$
be the optimal values of \reff{put-sdp:N} and \reff{put-sos:N} respectively.
From Theorem~\ref{las-M:dg-bd}, there exists an integer $N^*$
such that for all $\eps >0$
\[
f(x) - (f^*-\eps) \quad \in \quad I^{(N^*)} + M^{(N^*)}.
\]
Like in the proof of Theorem~\ref{thm:sos-exact}, we can similarly prove
$f_{N}^{(1)} =  f_{N}^{(2)} = f^*$ for all $N\geq N^*$.
Since AC holds, the set $S$ must be compact.
So the minimum $f_{min}$ of \reff{pop:gen} must be achievable.
By Assumption~\ref{as:fin+smth}, we know $f^* = f_{min}$,
and the proof is complete.
\qed

\subsection{A simplified version for inactive constraints}

Suppose in \reff{pop:gen} we are only interested in
a minimizer making all the inequality constraints inactive.
Consider the problem
\be  \label{pop:g>0}
\baray{rl}
\underset{x\in \re^n}{\min} & f(x) \\
\mbox{s.t.} & h_1(x) = \cdots = h_{m_1}(x) = 0, \\
& g_1(x) > 0, \ldots, g_{m_2}(x) > 0.
\earay
\ee
Let $u$ be a minimizer of \reff{pop:g>0}.
If $V(h)$ is smooth at $u$, there exist $\lmd_i$ such that
\[
\nabla f(u) \, = \lmd_1 \nabla h_1(u) + \cdots + \lmd_{m_1} \nabla h_{m_1}(u).
\]
Thus, $u$ belongs to the determinantal variety
\[
G_h = \left\{ x: \rank \bbm \nabla f(x) & \nabla h(x)  \ebm \leq m_1 \right\}.
\]
If $m_1 < n$,
let $\phi_1,\ldots,\phi_s$ be a minimum set of defining polynomials for $G_h$
by using formula \reff{def:eta_ell}.
If $m_1 = n$, then $G_h = \re^n$ and we do not need these polynomials;
set $s=0$, and $[s]$ is empty.
Then, \reff{pop:g>0} is equivalent to
\be  \label{ktvar:g>0}
\baray{rl}
\underset{x\in \re^n}{\min} & f(x) \\
\mbox{s.t.} & h_i(x) = 0, \phi_j(x) = 0, \, i\in [m_1], j\in [s], \\
& g_1(x) > 0, \ldots, g_{m_2}(x)>0.
\earay
\ee
The difference between \reff{ktvar:g>0} and \reff{pop:kktvar} is that
the number of new equations in \reff{ktvar:g>0} is
$s = O\Big(nm_1\Big)$, which is much smaller than $r$ in \reff{pop:kktvar}.
So, \reff{ktvar:g>0} is preferable to \reff{pop:kktvar}
when the inequality constraints are all inactive.
The $N$-th order Lasserre's relaxation for \reff{ktvar:g>0} is
\be  \label{las-N:g>0}
\baray{rl}
\min  &  L_f(y) \\
\mbox{s.t.} & L_{h_i}^{(N)}(y)= 0, L_{\phi_j}^{(N)}(y)= 0, \, i\in [m_1], j\in [s], \\
& L_{g_j}^{(N)}(y) \succeq 0, \, j=1,\ldots,m_2, \, y_0 = 1.
\earay
\ee
A tighter version than the above using cross products of $g_j$ is
\be  \label{schm-N:g>0}
\baray{rl}
\min  &  L_f(y) \\
\mbox{s.t.} & L_{h_i}^{(N)}(y)= 0, L_{\phi_j}^{(N)}(y)= 0, \, i\in [m_1], j\in [s], \\
& L_{g_\nu}^{(N)}(y) \succeq 0, \, \forall \nu \in \{0,1\}^{m_2}, \, y_0 = 1.
\earay
\ee
Define the truncated ideal $J^{(N)}$ generated by $h_i(x)$ and $\phi_j$ as
\[
J^{(N)} = \left\{
\sum_{i=1}^{m_1} p_i(x) h_i(x) + \sum_{j=1}^{s} q_j(x) \phi_j(x):
\baray{c}
\deg(p_ih_i) \leq 2N \quad \forall\, i\\
deg(q_j\phi_j) \leq 2N \quad \forall\,j
\earay
\right\}.
\]
The dual of \reff{las-N:g>0} is the SOS relaxation
\be \label{sos-las:g>0}
\baray{rl}
\max  & \gamma   \\
\mbox{s.t.} & f(x) - \gamma \in J^{(N)}+ M^{(N)}.
\earay
\ee
The dual of \reff{schm-N:g>0} is the SOS relaxation
\be \label{sch-sos:g>0}
\baray{rl}
\max  & \gamma   \\
\mbox{s.t.} & f(x) - \gamma \in J^{(N)}+ P^{(N)}.
\earay
\ee

The exactness of the above relaxations
is summarized as follows.

\begin{theorem} \label{thm:g>0}
Suppose the variety $V(h)$ is nonsingular
and the minimum $f_{min}$ of \reff{pop:g>0} is achieved
at some feasible $u$ with every $g_j(u)>0$.
If $N$ is big enough, then the optimal values of
\reff{schm-N:g>0} and \reff{sch-sos:g>0} are equal to $f_{min}$.
If, in addition, the archimedean condition holds for $S$, the optimal values of
\reff{las-N:g>0} and \reff{sos-las:g>0} are also equal to $f_{min}$
for $N$ big enough.
\end{theorem}
\begin{proof}
The proof is almost same as for
Theorems~\ref{thm:sos-exact} and \ref{thm:las-exact}.
We can first prove a decomposition result like Lemma~\ref{lm:W-dcmp},
and then prove there exists $N^*>0$ such that
for all $\eps > 0$ (like in Theorem~\ref{pro:unf-deg})
\[
f(x) - f_{min} + \eps \, \in \, J^{(N^*)} + P^{(N^*)}.
\]
Furthermore, if AC holds, we can similarly prove there exists $N^*>0$ such that
for all $\eps >0$ (like in Theorem.~\ref{las-M:dg-bd})
\[
f(x) - f_{min} + \eps \, \in \, J^{(N^*)} + M^{(N^*)}.
\]
The rest of the proof is almost same
as for Theorems~\ref{thm:sos-exact} and \ref{thm:las-exact}.
Due to its repeating, we omit the details here for the cleanness of the paper.
\end{proof}

\section{Examples} \label{sec:exmp}
\setcounter{equation}{0}

This section presents some examples on
how to apply the SDP relaxation \reff{las-sdp:deg=2N} and its dual \reff{sos:deg=2N}
to solve polynomial optimization problems.
The software {\it GloptiPoly~3} \cite{GloPol3}
is used to solve \reff{las-sdp:deg=2N} and \reff{sos:deg=2N}.

First, we consider an unconstrained optimization.

\begin{exm} \label{em-unc:0fail}
Consider problem
\[
\min_{x\in\re^3} \quad
x_1^8+x_2^8+x_3^8+x_1^4x_2^2+x_1^2x_2^4+x_3^6-3x_1^2x_2^2x_3^2.
\]
This example was studied in \cite{NDS}.
Its global minimum is zero.
We apply SDP relaxation \reff{las-sdp:deg=2N} of order $N=4$,
and get a lower bound $-9.7 \times 10^{-9}$.
The minimizer $(0,0)$ is extracted.
In \cite{NDS}, it was shown that $f(x)$ is not SOS modulo its gradient ideal $I_{grad}$.
But for every $\eps >0$, $f(x)+\eps \equiv s_\eps(x)$ modulo $I_{grad}$
for some SOS $s_\eps(x)$,
whose degree is independent of $\eps$ (see equation (10) of \cite{NDS}).
But its coefficients go to infinity as $\eps \to 0$.
This shows that the optimal value of \reff{sos:deg=2N}
might not be achievable.
\qed
\end{exm}

Second, we consider polynomial optimization having only equality constraints.
\be \label{opt:eq}
\underset{x\in \re^n}{\min} \quad f(x) \quad
\mbox{s.t.} \quad h_1(x) = \cdots = h_{m}(x) = 0.
\ee
When $V(h)$ is nonsingular, its equivalent version \reff{pop:kktvar} reduces to
\be \label{kkt-opt:g=0}
\baray{rl}
\underset{x\in \re^n}{\min} & f(x) \\
\mbox{s.t.} & h_1(x) = \cdots = h_{m}(x) = 0, \\
& \, \underset{ \substack{ I \in [n]_{m+1} \\  sum(I) = \ell}  }{\sum} \mbox{det}_I F(x) = 0,
\quad \ell = \binom{m+2}{2}, \ldots, (n-\frac{m}{2})(m+1).
\earay
\ee
In the above $sum(I)$ denotes the summation of the indices in $I$,
$F(x)= \bbm \nabla f(x) & \nabla h(x)\ebm$, and $\det_I F(x)$
denotes the maximal minor of $F(x)$ whose row indices are in $I$.
When $m \geq n$, there are no minor equations in \reff{kkt-opt:g=0}.

\begin{exm}
Consider the optimization
\begin{align*}
\min_{x\in\re^3} & \quad x_1^6+x_2^6+x_3^6 + 3x_1^2x_2^2x_3^2
-3(x_1^2(x_2^4+x_3^4)+x_2^2(x_3^4+x_1^4)+x_3^2(x_1^4+x_2^4)) \\
\mbox{s.t.} &  \quad x_1+x_2+x_3 - 1 =0.
\end{align*}
The objective is the Robinson polynomial,
which is nonnegative everywhere but not SOS \cite{Rez00}.
So the minimum $f_{min}=0$.
We apply SDP relaxation \reff{las-sdp:deg=2N} of order $N=4$,
and get a lower bound $-4.4600\times 10^{-9}$.
The minimizer $(1/3,1/3,1/3)$ is also extracted.
Applying Lasserre's relaxation \reff{sos:Put} of orders $N=3,4,5,6,7$,
we get lower bounds respectively
\[
-0.0582, \quad -0.0479, \quad -0.0194, \quad -0.0053, \quad -4.8358\times 10^{-5}.
\]
We can see that \reff{sos:Put} is weaker than \reff{las-sdp:deg=2N}.
It is not clear whether the sequence of relaxations \reff{sos:Put}
converges or not for this problem,
since the feasible set is noncompact.
But, the $f(x)$ here is not SOS modulo the constraint in this example.
Otherwise, suppose there exist polynomials $\sig(x)$ being SOS and $\phi(x)$ such that
\[
f(x) = \sig(x) + \phi(x)(x_1+x_2+x_3 - 1).
\]
In the above, replacing every $x_i$ by $x_i/(x_1+x_2+x_3)$ gives
\[
f(x) = (x_1+x_2+x_3)^6 \sig(x/(x_1+x_2+x_3)).
\]
So, there exist polynomials $p_1,\ldots, p_k, q_1,\ldots, q_\ell$ such that
\[
f(x) = p_1^2+\cdots+p_k^2 + \frac{q_1^2}{(x_1+x_2+x_3)^2} + \cdots +
\frac{q_\ell^2}{(x_1+x_2+x_3)^{2\ell}}.
\]
Since the objective $f(x)$ does not have any pole,
every $q_i$ must vanish whenever $x_1+x_2+x_3=0$.
Thus $q_i=(x_1+x_2+x_3)^{i}w_i$ for some polynomials $w_i$.
Hence, we get
\[
f(x) = p_1^2+\cdots+p_r^2 + w_1^2 + \cdots + w_\ell^2
\]
is SOS, which is a contradiction.
\qed
\end{exm}

Third, consider polynomial optimization having only a single inequality constraint.

\be \label{opt:one-eq}
\underset{x\in \re^n}{\min} \quad f(x) \quad
\mbox{s.t.} \quad g(x) \geq 0.
\ee
Its equivalent form \reff{pop:kktvar} becomes
\be
\baray{rl}
\underset{x\in \re^n}{\min} & f(x) \\
\mbox{s.t.} & g(x) \frac{\pt f(x)}{\pt x_i} = 0, \quad
\, i = 1,\ldots,n, \quad g(x) \geq 0, \\
& \underset{i+j=\ell}{\sum} \left( \frac{\pt f(x)}{\pt x_i} \frac{\pt g(x)}{\pt x_j} -
 \frac{\pt f(x)}{\pt x_j} \frac{\pt g(x)}{\pt x_i} \right) = 0, \quad
 \ell = 3, \ldots, 2n-1.
\earay
\ee
There are totally $3(n-1)$ equalities and a single inequality.

\begin{exm}  \label{em:Mzkn-ball}
Consider the optimization
\begin{align*}
\min_{x\in\re^3} & \quad x_1^4x_2^2+x_1^2x_2^4+x_3^6-3x_1^2x_2^2x_3^2 \\
\mbox{s.t.} &  \quad x_1^2+x_2^2+x_3^2 \leq 1.
\end{align*}
The objective is the Motzkin polynomial
which is nonnegative everywhere but not SOS \cite{Rez00}.
So its minimum is $0$.
We apply SDP relaxation \reff{las-sdp:deg=2N} of order $N=4$,
and get a lower bound $-1.6948 \times 10^{-8}$.
The minimizer $(0,0,0)$ is also extracted.
Now we apply Lasserre's relaxation \reff{sos:Put}.
For orders $N=4,5,6,7,8$, \reff{sos:Put} returns the lower bounds respectively
\[
-2.0331 \times 10^{-4},  -2.9222 \times 10^{-5},
 -8.2600 \times 10^{-6},  -4.2565 \times 10^{-6},
 -2.3465 \times 10^{-6}.
\]
We can see that \reff{sos:Put} is weaker than \reff{las-sdp:deg=2N}.
The sequence of \reff{sos:Put} certainly converges since the feasible set is compact.
However, the objective does not belong to the preordering
generated by the ball condition.
This fact was kindly pointed out to the author by Claus Scheiderer
(implied by his proof of Prop.~6.1 in \cite{Sch99},
since the objective is a nonnegative but non-SOS form vanishing at origin).
\qed
\end{exm}

\begin{exm}
Consider Example~\ref{em:Mzkn-ball}
but the constraint is the exterior of the ball:
\begin{align*}
\min_{x\in\re^3} & \quad x_1^4x_2^2+x_1^2x_2^4+x_3^6-3x_1^2x_2^2x_3^2 \\
\mbox{s.t.} &  \quad x_1^2+x_2^2+x_3^2 \geq 1.
\end{align*}
Its minimum is still $0$.
We apply SDP relaxation \reff{las-sdp:deg=2N} of order $N=4$,
and get a lower bound $1.7633 \times 10^{-9}$
(its sign is not correct due to numerical issues).
Now we compare it with Lasserre's relaxation \reff{sos:Put}.
When $N=4$, \reff{sos:Put} is not feasible.
When $N=5,6,7,8$, \reff{sos:Put} returns the following lower bounds respectively
\[
-4.8567 \times 10^{5}, \quad -98.4862, \quad -0.7079,  \quad    -0.0277.
\]
So we can see \reff{sos:Put} is much weaker than \reff{las-sdp:deg=2N}.
It is not clear whether \reff{sos:Put} converges or not for this problem,
since its feasible set is unbounded.
\qed
\end{exm}

Last, we show some general examples.

\begin{exm}
Consider the following polynomial optimization
\[
\baray{rl}
\underset{x\in\re^2}{\min} & \quad x_1^2+x_2^2 \\
\mbox{s.t.} & x_2^2 -1 \geq 0, \\
& x_1^2-Mx_1x_2 -1 \geq 0, \\
& x_1^2+Mx_1x_2 -1 \geq 0. \\
\earay
\]
This problem was studied in \cite{DNP,HLNZ}.
Its global minimum is $2+\half M(M+\sqrt{M^2+4})$.
Let $M=5$ here.
Applying \reff{las-sdp:deg=2N} of order $N=4$,
we get a lower bound $27.9629$ which equals the global minimum, and
four global minimizers $(\pm 5.1926, \pm 1.0000)$.
However, if we apply the Lasserre's relaxation either \reff{sos:Put} or \reff{sos:Smg},
the best lower bound we would obtain is $2$,
no matter how big the relaxation order $N$ is (see Example~4.5 of \cite{DNP}).
\qed
\end{exm}

\begin{exm}
Consider the polynomial optimization
\[
\baray{rl}
\underset{x\in\re^3}{\min} & x_1^4x_2^2+x_2^4x_3^2+x_3^4x_1^2-3x_1^2x_2^2x_3^2 \\
\mbox{s.t.} & 1-x_1^2 \geq 0, 1-x_2^2 \geq 0, 1-x_3^2 \geq 0.
\earay
\]
The objective is a nonnegative form being non-SOS \cite[Sec.~4c]{Rez00}.
Thus its minimum is $0$.
We apply SDP relaxation \reff{las-sdp:deg=2N} of order $N=6$,
and get a lower bound $-9.0752\times 10^{-9}$.
A minimizer $(0,0,0)$ is also extracted.
Now we apply Lasserre's relaxation of type \reff{sos:Smg}.
For $N=6,7,8$, \reff{sos:Smg} returns the lower bounds respectively
\[
-3.5619\times 10^{-5}, \quad -1.0406\times 10^{-5},  \quad  -7.6934\times 10^{-6}.
\]
So we can see that \reff{sos:Smg} is weaker than \reff{las-sdp:deg=2N}.
They converge in this case, since the feasible set is compact.
However,
the objective does not belong to the preordering generated by the constraints,
which is implied by the proof of Prop.~6.1 of \cite{Sch99}
(the objective is a nonnegative but non-SOS form vanishing at origin).
\qed
\end{exm}

\section{Some conclusions and discussions} \label{sec:con-dis}
\setcounter{equation}{0}

This paper proposes the exact SDP relaxation \reff{las-sdp:deg=2N}
and its dual \reff{sos:deg=2N}
for polynomial optimization \reff{pop:gen}
by using the Jacobian of its defining polynomials.
Under some generic conditions,
we showed that the minimum of \reff{pop:gen}
would be found by solving the SDP \reff{las-sdp:deg=2N}
for a finite relaxation order.

The results of this paper improve the earlier work \cite{DNP,NDS},
where the exactness of gradient or KKT type SOS relaxations
for a finite relaxation order is only proved when
the gradient or KKT ideal is radical.
There are other conditions like boundary hessian condition (BHC)
guaranteeing this property, like in \cite{Hiep10,Mar09}.
In \cite{Mar09}, Marshall showed that the gradient SOS relaxation is also
exact for a finite relaxation order by assuming BHC, in unconstrained optimization.
In \cite{Hiep10}, Hiep proposed a KKT type SOS relaxation using critical variety
for constrained optimization,
and its exactness for a finite relaxation order is also presented under BHC.
In this paper, the exactness of \reff{las-sdp:deg=2N} and \reff{sos:deg=2N}
for a finite $N$ is proved without the conditions like radicalness or BHC.
The only assumptions required are nonsingularity of $S$ and
the minimum $f_{min}$ being achievable
(the earlier related work also requires this),
but they are generically true as shown by Theorem~\ref{thm:con-gen}.

We would like to point out that the KKT type SOS relaxation
proposed in \cite{DNP} using Lagrange multipliers
is also exact for a finite order,
no matter the KKT ideal is radical or not.
This would be proved in a similar way as we did in Section~\ref{sec:proof}.
First, we can get a similar decomposition for the KKT variety
like Lemma~\ref{lm:W-dcmp}.
Second, we can prove a similar representation
for $f(x)-f^*+\eps$ like in Theorem~\ref{pro:unf-deg},
with degree bounds independent of $\eps$.
Based on these two steps, we can similarly prove 
its exactness for a finite relaxation order.
Since the proof is almost a repeating of Section~\ref{sec:proof},
we omit it for the cleanness of the paper.

The proof of the exactness of \reff{las-sdp:deg=2N}
provides a representation of polynomials that are positive on $S$
through using the preordering of $S$ and
the Jacobian of all the involved polynomials.
A nice property of this representation is that
the degrees of the representing polynomials
are independent of the minimum value.
This is presented by Theorem~\ref{pro:unf-deg}.
A similar representation result using the quadratic module of $S$
is given by Theorem~\ref{las-M:dg-bd}.

An issue that is not addressed by the paper is that
the feasible set $S$ has singularities.
If a global minimizer $x^*$ of \reff{pop:gen} is singular on $S$,
then the KKT condition might no longer hold, and $x^* \not\in W$.
In this case, the original optimization \reff{pop:gen}
is not equivalent to \reff{pop:kktvar},
and the SDP relaxation \reff{las-sdp:deg=2N} might not
give a correct lower bound for $f_{min}$.
It is not clear how to handle singularities
generally in an efficient way.

Another issue that is not addressed by the paper is
the minimum $f_{min}$ of \reff{pop:gen} is not achievable,
which happens only if $S$ is noncompact.
For instance, when $S=\re^2$, the polynomial
$x_1^2+(x_1x_2-1)^2$ has minimum $0$
but it is not achievable.
If applying the relaxation \reff{las-sdp:deg=2N} for this instance,
we would not get a correct lower bound.
Generally, this case will not happen, as shown by items (d), (e) of Theorem~\ref{thm:con-gen}.
In unconstrained optimization,
when $f_{min}$ is not achievable,
excellent approaches are proposed in \cite{GEZ,HaPh,Swg06}.
It is an interesting future work to generalize them to constrained optimization.

An important question is for what concrete relaxation order $N^*$
the SDP relaxation \reff{las-sdp:deg=2N} is exact for solving \reff{pop:gen}.
No good estimates for $N^*$ in Theorem~\ref{thm:sos-exact} are available currently.
Since the original problem \reff{pop:gen} is NP-hard,
any such estimates would be very bad if they exist.
This is another interesting future work.

\bigskip
\noindent
{\bf Acknowledgement} \, The author is grateful to Bernd Sturmfels
for pointing out the references on minimum
defining equations for determinantal varieties.
The author thanks Bill Helton for fruitful discussions.

\appendix
\section{Some basics in algebraic geometry and real algebra}

In this appendix,
we give a short review on basic algebraic geometry and real algebra.
More details would be found in the books \cite{CLO97,Ha}.

An ideal $I$ of $\re[x]$ is a subset such that $ I \cdot \re[x] \subseteq I$.
Given polynomials $p_1,\ldots,p_m \in \re[x]$,
$\langle p_1,\cdots,p_m \rangle $ denotes
the smallest ideal containing every $p_i$, which is the set
$p_1 \re[x] + \cdots + p_m \re[x]$.
The ideals in $\cpx[x]$ are defined similarly.
An algebraic variety is a subset of $\cpx^n$ that
are common complex zeros of polynomials in an ideal.
Let $I$ be an ideal of $\re[x]$. Define
\begin{align*}
V(I) &= \{x\in \cpx^n: \, p(x) = 0 \, \quad \forall \, p \in I\},  \\
V_\re(I) &= \{x\in \re^n: \, p(x) = 0 \, \quad \forall \, p \in I\}.
\end{align*}
The $V(I)$ is called an algebraic variety or just a variety,
and $V_\re(I)$ is called a real algebraic variety or just a real variety.
Every subset $T \subset \cpx^n$ is contained in a variety in $\cpx^n$.
The smallest one containing $T$
is called the {\it Zariski} closure of $S$, and is denoted by $Zar(T)$.
In the Zariski topology on $\cpx^n$, the varieties are called closed sets,
and the complements of varieties are called open sets.
A variety $V$ is irreducible if there exist no
proper subvarieties $V_1,V_2$ of $V$ such that $V= V_1 \cup V_2$.
Every variety is a finite union of irreducible varieties.

%

\begin{theorem} [Hilbert's Strong Nullstellensatz] \label{Null-Strong}
Let $I\subset \re[x]$ be an ideal. If $p\in \re[x]$ vanishes on $V(I)$,
then $p^k\in I$ for some integer $k>0$.
\end{theorem}

If an ideal $I$ has empty variety $V(I)$, then $1\in I$.
This is precisely the Hilbert's weak Nullstellensatz.
\begin{theorem} [Hilbert's Weak Nullstellensatz] \label{Null-Weak}
Let $I\subset \re[x]$ be an ideal. If $V(I) = \emptyset$,
then $1 \in I$.
\end{theorem}

Now we consider $I$ to be an ideal generated by polynomials having real coefficients.
Let $T$ be a basic closed semialgebraic set.
There is a certificate for $V_\re(I) \cap T = \emptyset$.
This is the so-called Positivstellensatz.

\begin{theorem} [Positivstellensatz, \cite{Sten}] \label{Pos-Nul}
Let $I \subset \re[x] $ be an ideal,
and $T=\{x\in \re^n:\, g_1(x) \geq 0, \ldots, g_r(x) \geq 0 \}$
be defined by real polynomials $g_i$.
If $V_\re(I) \cap T = \emptyset$, then there exist SOS polynomials $\sig_\nu$ such that
\[
-1 \equiv \sum_{ \nu \in \{0,1\}^r } \, \sig_\nu \cdot g_1^{\nu_1} \cdots g_r^{\nu_r}
\quad \mbox{ mod } \quad I.
\]
\end{theorem}

\begin{theorem} [Putinar's Positivstellensatz, \cite{Put}] \label{Put-Pos}
Let $I$ be an ideal of $\re[x]$ and
$T=\{x\in \re^n:\, g_1(x) \geq 0, \ldots, g_r(x) \geq 0 \}$
be defined by real polynomials $g_i$.
Suppose there exist $R>0$ and SOS polynomials $s_0(x),\ldots, s_m(x)$ such that
(the archimedean condition holds)
\[
R - \|x\|_2^2 \equiv s_0(x)+s_1(x)g_1(x)+\cdots+s_m(x)g_m(x)
\quad \mbox{ mod  } \quad I.
\]
If a polynomial $f(x)$ is positive on $V_\re(I) \cap T$,
then there exist SOS polynomials $\sig_i$ such that
\[
f(x) \equiv  \sig_0(x)+\sig_1(x)g_1(x)+\cdots+\sig_m(x)g_m(x)
\quad \mbox{ mod } \quad I.
\]
\end{theorem}

\bigskip

In the following, we review some elementary background about resultants and discriminants.
More details would be found in \cite{CLO98,GKZ,Stu02}.

Let $f_1,\ldots, f_n$ be homogeneous polynomials in $x=(x_1,\ldots,x_n)$.
The resultant $Res(f_1,\ldots, f_n)$ is a polynomial
in the coefficients of $f_1, \ldots,f_n$ satisfying
\[
Res(f_1,\ldots,f_n) = 0 \quad  \Longleftrightarrow \quad
\exists \, 0 \ne u\in \cpx^n, \, f_1(u)=\cdots = f_n(u)=0.
\]
The resultant $Res(f_1,\ldots, f_n)$ is homogeneous, irreducible and has integer coefficients.
When $f(x)$ is a single homogeneous polynomial,
its discriminant is defined to be
\[
\Delta(f) \, = \, Res(\frac{\pt f}{\pt x_1},\ldots, \frac{\pt f}{\pt x_n}).
\]
Thus, we have the relation
\[
\Delta(f) = 0 \quad  \Longleftrightarrow \quad
\exists \, 0 \ne u\in \cpx^n, \, \nabla f(u)=0.
\]

The discriminants and resultants are also defined for inhomogeneous polynomials.
Let $f_0,f_1,\ldots, f_n$ be general polynomials in $x = (x_1,\ldots,x_n)$.
Their resultant $Res(f_0,f_1,\ldots, f_n)$ is then defined to be
$Res(\tilde{f_0}(\tilde{x}), \tilde{f_1}(\tilde{x}), \ldots, \tilde{f_n}(\tilde{x}))$,
where each $\tilde{f_i}(\tilde{x})=x_0^{\deg(f_i)}f(x/x_0)$ is the homogenization of $f_i(x)$.
Clearly, if the polynomial system
\[
f_0(x)=f_1(x)=\cdots=f_n(x)=0
\]
has a solution in $\cpx^n$, then the homogeneous system
\[
\tilde{f_0}(\tilde{x})=\tilde{f_1}(\tilde{x})
=\cdots=\tilde{f_n}(\tilde{x})=0
\]
has a nozero solution in $\cpx^{n+1}$, and hence $Res(f_0,f_1,\ldots, f_n)=0$.
The reverse is not always true,
because the latter homogeneous system might have a solution at infinity $x_0=0$.
If $f(x)$ is a single nonhomogeneous polynomial,
its discriminant is defined similarly as $\Delta(\tilde{f})$.

The discriminants are also defined for several polynomials.
More details are in \cite[Sec.~3]{NieDis}.
Let $f_1(\tilde{x}),\ldots, f_m(\tilde{x})$
be forms in $x=(x_1,\ldots,x_n)$ of degrees $d_1,\ldots,d_m$ respectively,
and $m \leq n-1$. Suppose at least one $d_i>1$.
The discriminant for $f_1,\ldots,f_m$, denoted by $\Delta(f_1,\ldots,f_m)$,
is a polynomial in the coefficients of $f_i$ such that
\[
\Delta(f_1,\ldots,f_m) \, = \, 0
\]
if and only if the polynomial system
\[
f_1(x) = \cdots = f_m(x) = 0
\]
has a solution $u \ne 0$ such that the matrix
$\bbm \nabla f_1(u) &  \cdots & \nabla f_m(u) \ebm$
does not have full rank.
When $m=1$, $\Delta(f_1,\ldots,f_m)$ reduces to
the standard discriminant of a single polynomial.

When $f_1,\ldots,f_m$ are nonhomogeneous polynomials in $x=(x_1,\ldots,x_n)$ and $m\leq n$,
the discriminant $\Delta(f_1,\ldots,f_m)$ is then defined to be
$\Delta(\tilde{f_1}(\tilde{x}), \ldots,\tilde{f_m}(\tilde{x}))$,
where each $\tilde{f_i}(\tilde{x})$ is the homogenization of $f_i(x)$.

\end{document}